\documentclass[11pt,a4paper]{amsart}
\numberwithin{equation}{section}
\usepackage{rotating}
\usepackage{mathrsfs}
\usepackage{amscd,amssymb,amsopn,amsmath,amsthm,graphics,amsfonts,enumerate,verbatim,calc}
\usepackage[all]{xy}
\usepackage[utf8]{inputenc}
\usepackage{color}

\usepackage[a4paper,margin=1in]{geometry}
\usepackage[utf8]{inputenc}
\usepackage{amscd,amssymb,amsopn,amsmath,amsthm,graphics,amsfonts,enumerate,verbatim,calc}
\usepackage[T1]{fontenc}
\usepackage{lmodern}
\usepackage{enumitem}
\usepackage{hyperref}
\usepackage{microtype}\usepackage{hyperref}
\usepackage{geometry}
\usepackage{enumitem}
\usepackage{booktabs}
\usepackage{tikz-cd}
\hypersetup{colorlinks=true,linkcolor=blue,citecolor=blue,urlcolor=blue}
\newtheorem{theorem}{Theorem}[section]
\newtheorem{proposition}[theorem]{Proposition}
\newtheorem{corollary}[theorem]{Corollary}
\newtheorem{lemma}[theorem]{Lemma}
\newtheorem{fact}[theorem]{Fact}\newtheorem{claim}[theorem]{Claim}
\newtheorem{example}[theorem]{Example}\newtheorem{notation}[theorem]{Notation}
\newtheorem{observation}[theorem]{Observation}
\newtheorem{remark}[theorem]{Remark}
\newtheorem{question}[theorem]{Question}
\newtheorem{conjecture}[theorem]{Conjecture}

\newtheorem{definition}[theorem]{Definition}

\newcommand{\rr}{\operatorname{r}}

\newcommand{\Spec}{\operatorname{Spec}}

\newcommand{\Gr}{\operatorname{Gr}}

\DeclareMathOperator{\edim}{edim}
\DeclareMathOperator{\length}{\ell}\DeclareMathOperator{\len}{\ell}
\newcommand{\m}{\mathfrak m}
\newcommand{\n}{\mathfrak n}
\newcommand{\fm}{\mathfrak m}
\newcommand{\fn}{\mathfrak n}
\newcommand{\p}{\mathfrak p}
\newcommand{\q}{\mathfrak q}
\newcommand{\Zge}{\mathbb Z_{\geq 0}}

\newcommand{\llength}{\ell\ell}
\newcommand{\gll}{\operatorname{g\ell\ell}}

\newenvironment{PROOF}[2][\proofname.]
{\begin{proof}[#1]\renewcommand{\qedsymbol}{\textsquare$_{\rm #2}$}}
	{\end{proof}}
\begin{document}
	
\author[M. Asgharzadeh]{Mohsen Asgharzadeh}
\date{}

\address{M. Asgharzadeh}
\email{mohsenasgharzadeh@gmail.com}

	\title[Reduction numbers under flatness]
{Reduction numbers under flat extensions}


	\begin{abstract}
 The reduction number of the maximal ideal of a localization (or even faithfully flat extension) need not
 increase under specialization.   Nevertheless, for
 each fixed local ring the reduction numbers of the maximal
 ideals of all its localization  have a finite common bound.  A connection to (generalized) Loewy length is given.
\end{abstract}

	\maketitle

	\section{Introduction}
	\label{sec:intro}
	
 If $I$ is an ideal and $J\subseteq I$,
	then $J$ is called  a \emph{reduction} of $I$ if
$
	I^{a+1}=JI^a$  for some $a\in\Zge.$
	The least such $a$ is denoted by $r_J(I)$. A minimal reduction is a
	reduction minimal under inclusion, and throughout this note
$$
		\rr(I):=\min\{r_J(I): J\text{ is a minimal reduction of }I\}.
$$
	The restriction to minimal reductions matters: allowing every
	reduction in this minimum would give zero by taking $J=I$.
In the local setting, it controls the Hilbert--Samuel
	function, the structure of the associated graded ring, and the
	Cohen--Macaulayness of blow-up algebras. 
		There is an extensive literature on bounds for reduction numbers. In particular, Vasconcelos obtained general bounds for reduction numbers of ideals and multiplicity-dependent bounds for primary ideals. Related estimates have been developed using the Eakin–Sathaye theorem and, under additional hypotheses, the theory of Sally modules and associated graded rings. For more details, see the book  of Vasconcelos \cite{Vasconcelos}.
	A natural question, going back to
	Vasconcelos, is how this invariant behaves under localization, and
	whether it is uniformly bounded on the prime spectrum of a fixed
	ring. 
	
	\begin{question}[Vasconcelos, see {\cite[Page 83]{Vasconcelos}}]\label{q:vasconcelos}
		Let $R$ be  local  and let $\p\subsetneq\q$ be prime
		ideals of $R$. 
		\begin{enumerate}[label=\textup{(\roman*)}]
			\item Is
			$\rr(\p R_{\p})\leq \rr(\q R_{\q})$?
			\item Can
			$\sup\{\rr(\p R_{\p}) : \p\in\Spec(R)\}$
			be infinite?
		\end{enumerate}
	\end{question}

	Part (i) asks whether reduction numbers are \emph{monotone} under
	specialization, i.e.\ whether passing from the larger prime $\q$ to
	the smaller prime $\p$ can only decrease the reduction number.
	Part (ii) asks whether, for a fixed Noetherian local ring, the
	reduction numbers of the maximal ideals of all its localizations
	admit a  common finite bound, independently of the prime.
		Recall that a positive answer to (i) gives a positive answer to (ii), and not necessarily the converse.

	\begin{theorem}\label{thm:answers}

		\begin{enumerate}[label=\textup{(\roman*)}]
			\item The implication
			$
		``\p\subsetneq\q
		\Longrightarrow\
			\rr(\p R_{\p})\leq \rr(\q R_{\q})''
			$
			is not the case.
			\item  
			$
			\sup_{\p\in\Spec R}\rr(\p R_{\p})<\infty.
			$
	
		\end{enumerate}
	\end{theorem}
	 Without the noetherian hypothesis, $
	\sup_{\p\in\Spec R}\rr(\p R_{\p})
	$ can be infinite,
	even when $\dim R=1$, and every term in the supremum is finite (see  Section~\ref{sec:nonnoetherian}). 
Section~\ref{2} deals with
the following question from the exercise \cite[1.193]{Vasconcelos}  attributed to Hanes (see \cite[Conjecture 3.1]{Han}):
\begin{question}
	Let $f\colon(R,\fm)\to(S,\fn)$ be a flat local morphism. Is $\rr(\fm)\leq\rr(\fn)$?
\end{question}

 We present progress on the monotonicity property of the reduction number with respect to faithfully flat extensions, under additional Cohen–Macaulay hypotheses (see Proposition~\ref{p24}). In Proposition~\ref{fh}, we establish a positive answer under suitable graded assumptions. We note that this was recently   proved by  Pollitz   and  KC \cite{NP}  via a different argument; our proof appears simpler and yields a sharp bound, although the assumptions in \cite{NP} are sharper than ours. On the negative side, we present Proposition~\ref{e25}. This is inspired by \cite{HSh}, and of course the motivational example from the book of Huneke-Swanson \cite{HS}.

In
	Section~\ref{sec:counterexample} we present  
	one-dimensional Cohen--Macaulay rings with
	$\rr(\m)=1$ and $\rr(\p R_{\p})=N$ for every $N\geq2$.
The proof of Theorem~\ref{thm:answers}(ii) occupies Section ~\ref{sec:uniform} and combines three ingredients. The first is
	an elementary estimate, proved in Section~\ref{sec:colength}, which
	bounds the number of generators of a power $\n^h$ of the maximal
	ideal of a $d$-dimensional local ring in terms of $h$, $d$,
	and the colength of a parameter ideal; this estimate does not
	require the Cohen--Macaulay property. Combined with the
	Eakin--Sathaye theorem, it yields a local bound
	$\rr(\n)\leq d!\,\ell_A(A/Q)-1$ for every parameter ideal $Q$,
	valid whenever the residue field is infinite. The second ingredient,
	proved in Section~\ref{sec:uniform2}, is a uniform Brian\c
	con--Skoda statement for complete reduced local rings, which
	supplies a single exponent $C$ such that $\n^C\subseteq J$ for
 every prime $\p$ and  every reduction $J$ of
	$\p R_{\p}$. This converts the local colength bound into a
	uniform bound on the colength $\ell_A(A/Q)$, hence on
	$\rr(\p R_{\p})$. The third ingredient is a uniform
	embedding-dimension bound, obtained from the Cohen structure
	theorem and Lech's inequality for flat local maps, which makes the
	resulting bound depend only on the fixed ring $R$.

Recall from  previous remarks that the reduction number is connected with the generalized Loewy length; our reference for this is \cite{NP} and the references therein. Pollitz and Nawaj KC \cite{NP} considered the following, known as the Loewy Length Conjecture with a contribution to an unpublished manuscript of Corso, Huneke, Polini and Ulrich:

\begin{conjecture}\label{14}
 For a local ring $R$  and a nonzero $R$-module $M$ with finite projective dimension, the following inequality holds:
	
	$$\llength_R(M)\geq \gll(R):=\min\{\llength_R(R/(\underline{x})) : \underline{x} \text{ ranges over systems of parameters}\}.$$
	
\end{conjecture}

Here, $\llength_R(-)$ denotes the Loewy length. A conjectural connection between this and the reduction number was hinted to us by Pollitz and KC. It may be worth mentioning that

$$\m(R):=\min\{\length_R(R/(\underline{x})) : \underline{x} \text{ ranges over systems of parameters}\},$$
is also considered in the literature; see \cite{moh} and references therein, we call it Matsumura's number.
In the final section we show that  Conjecture \ref{14} needs some additional assumptions, for instance that the residue field is infinite. This follows by an analysis of 
a ring given  by
Hashimoto and Shida \cite[Example 3.2]{HSh}.

		\section{Conveniences}	\label{02}
	
	All rings are commutative with identity. Unless explicitly stated
	otherwise, rings are Noetherian. For a local ring $(R,\m)$, write $\mu_R(M)$ for the minimal number
	of generators of a finite module $M$, and $\ell_R(M)$ for its
	length when finite. For a finite-length module $N$, write $\llength_R(N)$ for the least
	$j\geq 0$ with $\m^jN=0$.

	\begin{remark}\label{rem:existence}
		The convention $I^{a+1}=JI^a$, taking the minimum over
		\emph{minimal} reductions, is essential to the formulation of the
		problem. With the larger class of all reductions the minimum would
		always be zero, since $J=I$ is a reduction of $I$, and both
		questions would trivialize. Minimal reductions exist for finitely
		generated ideals in Noetherian local rings regardless of the
		residue field; for completeness we recall the argument here. Indeed, minimal reductions exist for finitely generated ideals in Noetherian
		local rings, also when the residue field is finite. For $(S,\n)$ and an ideal $I$, choose
		a reduction $J$ for which
		$s=\dim_{S/\n}((J+\n I)/\n I)$ is as small as possible. Lift a
		basis to $y_1,\dots,y_s\in J$ and put $Q=(y_1,\dots,y_s)$.
		Since $J\subseteq Q+\n I$, a reduction equality for $J$, followed
		by Nakayama's lemma applied to $I^{a+1}/QI^a$, shows that $Q$ is
		a reduction. Independence of the chosen initial classes gives
		$Q\cap\n I=\n Q$. If $K\subseteq Q$ is a reduction, minimality
		of $s$ gives $K+\n I=Q+\n I$, hence $Q=K+\n Q$ and $K=Q$ by
		Nakayama. Thus $Q$ is minimal. In particular, the reduction number
		is a finite integer for the ideals considered in {Theorem} \ref{thm:answers}(ii).
	\end{remark}

The books of Matsumura \cite{Mat},  Huneke--Swanson \cite{HS} and Vasconcelos \cite{Vasconcelos} serve as
our standard references.
	
	\section{Dealing with an exercise}	\label{2}
	\subsection{Some positive side}	
We start with:

\begin{lemma}\label{art}
Let $f:(R,\fm)\to (S,\fn)$ be   flat and local morphism of artinian rings. Then $\rr(\fm)\leq \rr(\fn)$. The equality holds, if in addition $f$ is unramified (or at least $\fn=\fm S$).
\end{lemma}

\begin{proof}
	Let $r$ be an integer. The condition $\m^r=0$ is equivalent with $\m^r\otimes_RS=0$. But, $\m^r\otimes_RS\cong \fm^rS\subseteq \fn^r$.
Now, let $s:=\rr(\fn)$. Since the zero ideal is a minimal reduction of $\fn$, we know $\fn^{s+1}=0$. Let $r:=s+1$.	This, in turns, imply that 
	$\m^r=0$, and consequently, $\rr(\fm)\leq s=\rr(\fn)$.
	Now, suppose $\fn=\fm S$. Then $\fm^r=0$ iff $\fn^r=0$, and so  
$\rr(\fm)=\rr(\fn)$. 
\end{proof}

The second assertion is not true in the higher dimension:

\begin{example}\label{ex:residue-field-extension}
	Let
	$
	R=\mathbb F_2[[X,Y]]/(XY(X+Y))
	$
	and
	$
	S=\mathbf{F}[[X,Y]]/(XY(X+Y)),
	$
	where $\mathbf{F}$ denotes the algebraic closure of
	$\mathbb F_2$. Let $\mathfrak m=(X,Y)R$ and
	$\mathfrak n=(X,Y)S$. The natural homomorphism
	$R\to S$ is faithfully flat,
	$\mathfrak n=\mathfrak mS$ and 
		$
	\rr(\mathfrak m) \neq \rr(\mathfrak n).
	$
\end{example}
\begin{proof}	By   \cite[Example 8.3.2]{HS} (also see \cite{HSh}), the ideal $\mathfrak m$
	is basic: it has no proper reductions. Consequently, $\mathfrak m$
	is its own minimal reduction, and, with our convention,
	$
	\rr(\mathfrak m)=0.
	$
	On the other hand, choose $u\in\mathbf{F}$ with
	$u\ne 0,1$, and put $J=(X+uY)S$. The calculation in
\cite[Example 8.3.2]{HS} shows that
	$
	J\mathfrak n^2=\mathfrak n^3.
	$
	Moreover, $J$ is a minimal reduction of $\mathfrak n$.
	Since
	$
\Gr_{\mathfrak n}(S)
	\cong
\mathbf{F}[X,Y]/(XY(X+Y))
	$
	and
	$
	\frac{\Gr_{\mathfrak n}(S)}{(X+uY)}
	\cong
	\mathbf{F}[Y]/(Y^3),
	$
	the quotient has a nonzero degree-two component and vanishes in
	all degrees greater than two. Hence
	$
	0<\rr(\mathfrak n)\leq2,
	$ as $\fn$ is not basic.
\end{proof}

\begin{remark}\label{rm}
In general, the reduction number does not behave well with respect to a
superficial element $x\in I$; that is, the equality
$
 r_{J/(x)}\bigl(I/(x)\bigr)= r_J(I)
$
need not hold. However, when $\operatorname{depth} {\Gr}(I)\ge 1$, one has
$
 r_{J/(x)}\bigl(I/(x)\bigr)= r_J(I),
$
see \cite[Lemma 2.14]{mar}.		
\end{remark}

\begin{proposition}\label{p24} Let $f:(R,\mathfrak m)\to(S,\mathfrak n)$ be a flat local homomorphism of Cohen–Macaulay local rings of the same dimension, with infinite residue fields. 
	\begin{enumerate}[label=\textup{(\roman*)}]
		\item If $\operatorname{Gr}_{\mathfrak m}(R)$ is Cohen–Macaulay, then
		$
		\mathrm r(\mathfrak m)\leq \mathrm r(\mathfrak n).
		$
		\item Assume in addition to part i) that $\fn=\fm S$. Then 
		$
		\mathrm r(\mathfrak m)=\mathrm r(\mathfrak n).
		$
	\end{enumerate}

\end{proposition}

\begin{proof}
	i) 	Reduction by regular sequence shows that $\rr(\overline{\fm})\leq \rr(\fm)$, where $\overline R=R/(x_1,\ldots,x_d)$ and $\overline{\mathfrak m}=\mathfrak m\overline R$, with $x_1,\ldots,x_d$ a suitable superficial system of parameters for $R$. Similarly, $\rr(\overline{\fn})\leq \rr(\fn)$.
	Combining with	{Lemma} \ref{art} it is enough to show that $\rr(\overline{\fm})= \rr(\fm)$. This follows by the Cohen-Macaulay assumption of $\Gr_{\fm}(R)$ where we take reduction by $x$ so that  $x+I^2\in I/I^2=(\Gr_{\fm}(R))_1$  is $(\Gr_{\fm}(R))$-regular (see {Remark} \ref{rm}).

	ii) Since $\fn := \fm S$, and by flatness, we have 
	
	\[ \Gr_\fm(R) \otimes_R S \cong \bigoplus_n \frac{(\fm^n ) \otimes_R S}{(\fm^{n+1}  ) \otimes_R S} \cong \bigoplus_n \frac{\mathfrak{n}^n  S }{\mathfrak{n}^{n+1} S } = \Gr_{\mathfrak{n}}( S )\]This implies that
	$\Gr_{\mathfrak{n}}( S )$ is also Cohen-Macaulay. Then by    {Remark} \ref{rm}
	$
	\rr(\overline{\mathfrak n})= \rr(\mathfrak n).
	$ By arguing as part i), we have 		$
	\mathrm r(\mathfrak m)=\mathrm r(\mathfrak n).
	$
\end{proof}
\begin{proposition}\label{fh}
	Let $R\to S$ be a degree-preserving flat homomorphism of standard graded Cohen--Macaulay $k$-algebras over an infinite field $k$, where $R_0=S_0=k$ and $\dim R=\dim S$. Let $\mathfrak m=R_+$ and $\mathfrak n=S_+$.	\begin{enumerate}[label=\textup{(\roman*)}]
		\item One has	$
		\rr(\mathfrak m)\leq \rr(\mathfrak n).
		$
		\item Assume in addition to part i) that $\fn=\fm S$. Then 
		$
		\mathrm r(\mathfrak m)=\mathrm r(\mathfrak n).
		$
	\end{enumerate}

\end{proposition}

\begin{proof}
i)	Set $F=S/\mathfrak mS$. Since $R\to S$ is flat and the two rings have the same dimension, $F$ is zero-dimensional. Hence $F$ is an Artinian graded $k$-algebra.
	\begin{claim}\label{c28}
		One has $H_S(t)=H_R(t)H_F(t)$.
	\end{claim}
	\begin{PROOF}{\ref{c28}}
		Since $F=S/\mathfrak mS$ is a finite-dimensional graded $k$-vector space, choose
		a homogeneous $k$-basis
		$
		\{\overline{u}_1,\ldots,\overline{u}_N\}
		$
		with $\deg\overline{u}_i=a_i$. Lift each $\overline{u}_i$ to a homogeneous
		element $u_i\in S$ of the same degree.
		Because $S$ is finite and flat over $R$, these lifts form a homogeneous
		$R$-basis of $S$. Thus, as graded $R$-modules,
		$
		S\cong\bigoplus_{i=1}^N R(-a_i),
		$
		where $R(-a_i)$ denotes the graded shift of $R$.
		Taking Hilbert series gives
		\[
		\begin{aligned}
			H_S(t)
			=\sum_{i=1}^N H_{R(-a_i)}(t) 
			=\sum_{i=1}^N t^{a_i}H_R(t) 
			=H_R(t)\sum_{i=1}^N t^{a_i}.
		\end{aligned}
		\]
		The degrees of the homogeneous basis elements give
		$
		H_F(t)=\sum_{i=1}^N t^{a_i}.
		$
		The claim follows.
	\end{PROOF}

	As $R$ and $S$ are Cohen--Macaulay of the same dimension $d$, we can write
	\[
	H_R(t)=\frac{h_R(t)}{(1-t)^d},
	\qquad
	H_S(t)=\frac{h_S(t)}{(1-t)^d},
	\]
	where $h_R(t)$ and $h_S(t)$ are the respective $h$-polynomials. It follows that
	$
	h_S(t)=h_R(t)H_F(t).
	$
	Since $H_F(t)$ is a nonzero polynomial with nonnegative coefficients, we obtain
	$
	\deg h_S(t)\geq \deg h_R(t).	
	$	Finally, for a standard graded Cohen--Macaulay algebra over an infinite field, the reduction number of the homogeneous maximal ideal equals the degree of its $h$-polynomial. Thus
	\[
	\rr(\mathfrak m)=\deg h_R(t)
	\leq \deg h_S(t)=\rr(\mathfrak n).
	\]

ii) Under this additional assumption, $F$ is just reisdue field and so
$
h_S(t)=h_R(t)H_F(t)=h_R(t), 
$ and so the claim follows.\end{proof}

\begin{observation} One has
$\rr(\fm)=\rr(\hat{\fm})$.
\end{observation}

\begin{proof}
For any ideal $I$, consider the special fiber ring
$
	\mathcal F_R(I)
	=
	\bigoplus_{n\geq 0}\frac{I^n}{\mathfrak m I^n}.
$
	Since completion is faithfully flat, we have a natural graded
	isomorphism
$
	\mathcal F_R(I)
	\cong
	\mathcal F_{\widehat R}(I\widehat R)
$.	Recall from \cite[Proposition 8.2.4]{HS} that an ideal $J\subseteq I$ is a reduction of $I$ if
	and only if its image
$
	J^*=\frac{J+\mathfrak m I}{\mathfrak m I}
	\subseteq \mathcal F_R(I)_1
$
	satisfies
$
	\mathcal F_R(I)_{n+1}
	=
	J^*\mathcal F_R(I)_n
$
	for all sufficiently large $n$. Moreover,
	\[
	r_J(I)
	=
	\min\left\{
	n\geq 0:
	\mathcal F_R(I)_{n+1}
	=
	J^*\mathcal F_R(I)_n
	\right\}.
	\]
	The point is that $\hat{\fm}/\hat{\fm}^2=\fm/ \fm^2$, i.e., any reduction of $\hat{\fm}$ is algebraic, that it is extended from an ideal of $R$.
\end{proof}

\subsection{A negative side}

	The following  example uses finite residue fields and non-Cohen--Macaulay rings.

\begin{proposition}\label{e25}
	Let $k:=\mathbb F_2$, $K:=\mathbb F_4$. Let also $R:=\frac{k[[x,y]]}{(x^4,\;xy(x+y))} 
	\stackrel{f}\subseteq S:=\frac{K[[x,y]]}{(x^4,\;xy(x+y))}.$ Then $f$ is  flat and $\mathfrak n=\mathfrak mS$. Also,
${\rr(\mathfrak m)=3>2=\rr(\mathfrak n)}.
$

\end{proposition}

\begin{proof}
	For  a coefficient field $F$, the associated graded ring is
	$
	A_F=F[X,Y]/(X^4,\;XY(X+Y)).
	$ Recall that $f$ is local and finite free of rank $2$, hence flat, with
	$\mathfrak m=(x,y)$ and $\mathfrak n=\mathfrak mS$. Let
	$G:={\Gr}_{\mathfrak m}(A_F)
	\cong F[X,Y]/(X^4,XY(X+Y)),
$ see \cite[Ex. 14.2(iii)]{Mat}.
	Let $J$ be a minimal reduction of $\mathfrak m$. If $R/ \fm$ were be infinite, then
	every minimal reduction of  is minimally generated by exactly the analytic spread of $\fm$ elements (see \cite[Proposition 8.3.7]{HS}). Here, we have no such carnality assumption. Then, we argue directly, i.e.,  there exists an integer $r\ge 0$ such that
$
	\mathfrak m^{r+1}=J\mathfrak m^r.
$
	Consider the quotient map
$\pi:A_F\to A_F/(x)\cong F[[y]].$
	Since $\pi(\mathfrak m)=(y)$, applying $\pi$ to the reduction equality gives
$(y)^{r+1}=\pi(J)(y)^r.$
	Consequently, $\pi(J)=(y)$. Hence there exists $f\in J$ whose image modulo $x$ has a nonzero linear term in $y$. In particular, the initial form of $f$ is
$
	f^*=aX+bY$ with $a,b\in F$.
\begin{claim}\label{c27}
	One has $b\neq 0$.
\end{claim}
\begin{PROOF}{\ref{c27}}	First, $f\notin(x)$. Indeed, if $f\in(x)$, then $(f)\subseteq(x)$, and a reduction
	would require $\mathfrak m^{n+1}=(f)\mathfrak m^n$ for some $n$. Reducing modulo
	$x$ would give $(y)^{n+1}=0$ in $F[[y]]$, which is impossible. Hence $f\notin(x)$,
	so the image $g(y)$ of $f$ in
	$F[[x,y]]/(x)\cong F[[y]]$
	is nonzero.
	Next, $g(y)$ must generate the maximal ideal $(y)$ of the discrete valuation ring
	$F[[y]]$. Indeed, the image of $\mathfrak m$ modulo $x$ is $(y)$, and a reduction
	of $(y)$ in a DVR is generated by an element of $y$-adic valuation exactly $1$.
	Hence $g(y)$ has a nonzero linear term.
	Write
$
	f=a'x+b'y+ \text{terms of order}\ge2 ,
$
	so that $f^*=a'X+b'Y$. Reducing modulo $x$,
$
	g(y)\equiv b'y+ \text{higher order} \pmod{x}.
$
	Since $g(y)$ has nonzero linear term, we conclude $b'\neq0$. Thus $b\neq0$ in
	$f^*=aX+bY$.
\end{PROOF}	
	Since $b\ne0$, the relation $aX+bY=0$ expresses $Y$ as a scalar multiple of $X$. Therefore
$$	G/(f^*)\cong F[X,Y]/(X^4,XY(X+Y),aX+bY)$$
	is finite-dimensional over $F$, because it is a quotient of $F[X]/(X^4)$. Thus there exists an integer $N\ge0$ such that
$
	\bigl(G/(f^*)\bigr)_{n+1}=0$
 for all   $n\ge N,$
i.e.,
$
	G_{n+1}=f^*G_n.
$
By the definition of the associated graded ring, this implies
$
	\mathfrak m^{n+1}
	=f\mathfrak m^n+\mathfrak m^{n+2}.
$
	Set $M=\mathfrak m^{n+1}/f\mathfrak m^n$. The preceding equality yields $M=\mathfrak m M$. Since $M$ is a finitely generated $A_F$-module and $\mathfrak m$ is the maximal ideal, Nakayama's lemma gives $M=0$. Hence
$
	\mathfrak m^{n+1}=f\mathfrak m^n
.$
	It follows that $(f)$ is a reduction of $\mathfrak m$.
Finally, $f\in J$ implies $(f)\subseteq J$. Since $J$ is a minimal reduction and $(f)$ is itself a reduction, we conclude that
${J=(f).}$ 	This proves that every minimal reduction of $\mathfrak m$ is principal.

\begin{claim}\label{c26}
 	$
 r_{(f)}(\mathfrak m_F) =\max\{j:(A_F/(f^*))_j\ne0\}.
 	$
\end{claim}

\begin{PROOF}{\ref{c26}}	
Let $f\in\mathfrak m_F$ be a principal reduction of
$\mathfrak m_F$ so that $(f)$ is a minimal reduction. Write $f^*\in A_F$ for the
initial form of $f$ in the associated graded ring
$\operatorname{Gr}_{\mathfrak m_F}(A)=F[X,Y]/(X^4,\;XY(X+Y)).$
For a minimal reduction $J=(f)$ of $\mathfrak m_F$, the reduction number is
$r_J(\mathfrak m_F)=\min\{n:\mathfrak m_F^{n+1}=J\,\mathfrak m_F^{\,n}\}.	$
Passing to the associated graded ring, the condition $\mathfrak m_F^{n+1}=J\mathfrak m_F^n$
is equivalent to the surjectivity of the multiplication map
$\times f^*:(A_F)_{n-1}\to(A_F)_n,$
i.e.  to the vanishing
$(A_F/(f^*))_n=0.$Since $f$ is a reduction, $A_F/(f^*)$ is a finite-dimensional  artinian  graded
$F$-algebra, so it vanishes in all sufficiently high degrees. Because $f^*$ is a
nonzerodivisor of degree $1$, the quotient $(A_F/(f^*))_j$ is nonzero precisely for
$j\le r$, where $r$ is the largest index with nonzero component. Hence
$(A_F/(f^*))_j\ne0$  iff $j\le  r_{(f)}(\mathfrak m_F).$	
\end{PROOF}

Let $f\in\mathfrak m=(x,y)$ be such that $(f)$ is a minimal
reduction of $\mathfrak m$, and write $f^*=aX+bY$ for its initial form in
$A_F=F[X,Y]/(X^4,XY(X+Y))$.
Equivalently, $b\neq0$ says that $f$ is not contained in $(x)+\mathfrak m^2$: a
reduction of $\mathfrak m$ must detect the $y$-direction modulo $x$, and the only
way to do so at first order is to have a nonzero $Y$-coefficient in $f^*$.
Over $\mathbb F_2$, the only possible initial forms are $Y$ and $X+Y$, and
$
	A_k/(Y)\cong k[X]/(X^4)\cong A_k/(X+Y).
$
	Consequently, $\rr(\mathfrak m)=3$.
	Over $K$, choose $\alpha\notin\mathbb F_2$. Then
$A_K/(Y+\alpha X) \cong K[X]/\bigl(X^4,\alpha(1+\alpha)X^3\bigr) \cong K[X]/(X^3),$ because $\alpha(1+\alpha)$ has inverse over $K$.
	Thus $\rr(\mathfrak n)\le2$. Equality holds because $A_K$ has no quadratic relations: its degree-two component has dimension $3$, while multiplication by one linear form from its degree-one component has rank at most $2$.
\end{proof}

\section{Something is  not monotone}
\label{sec:counterexample}
Let $k$ be any field.
Let us compute reduction number in our  favorite non Cohen-Macaulay ring.
\begin{example}
	Let
	$
	R_0:=k[[x,y]]/(x^2,xy)$ with
	$\mathfrak m=(x,y)R_0.
	$
	Since $x^2=xy=0$, we have
	$
	\mathfrak m^2=(y^2)=y\mathfrak m.
	$
	Thus $J=(y)$ is a reduction of $\mathfrak m$. Moreover, $J$ is a minimal reduction, and
	$
 r_J(\mathfrak m)=1,
	$
	since $\mathfrak m\neq J$ but $\mathfrak m^2=J\mathfrak m$. Hence
	$
	\mathrm r(\mathfrak m)=1.
	$
	Now localize at the prime ideal $\mathfrak p=(x)$. Since $y\notin\mathfrak p$, the element $y$ is a unit in $(R_0)_{\mathfrak p}$. The relation $xy=0$ therefore implies $x=0$ in $(R_0)_{\mathfrak p}$. Consequently,
	$
	\mathfrak p(R_0)_{\mathfrak p}=(0).
	$
	Taking the zero ideal as a reduction of itself, we obtain
	$
	\mathrm r_{\mathfrak p(R_0)_{\mathfrak p}}
	(\mathfrak p(R_0)_{\mathfrak p})=0.
	$
	Thus
	$
	{
		\mathrm r(\mathfrak m)=1
		\text{ and } 
		\mathrm r_{\mathfrak p(R_0)_{\mathfrak p}}
		(\mathfrak p(R_0)_{\mathfrak p})=0.
	}
	$
\end{example}

 Our main ring in this section follows by adding a binomial relation $y^2-tx$ to  $R_0$. Namely, 
$ 
R:=k[[t,x,y]]/(x^2,xy,y^2-tx)$ (and its generalization), and we want to compare the reduction number of localization at 
$\p:=(x,y)$ with $\m:=(t,x,y).$
 
\begin{proposition}\label{prop:small}
The ring $R$ is a complete one-dimensional Cohen--Macaulay and local,
	$\p\subsetneq\m$, and
	$
	\rr(\p R_{\p})=2>1=\rr(\m).
	$
\end{proposition}

\begin{proof}
	Let $D:=k[[t]]$. The defining relations give a $D$-basis
	$1,x,y$ for $R$. For completeness, independence follows by realizing
	this algebra as the subalgebra with basis $1,tu^2,tu$ of
	$D[u]/(u^3)$, with $x\mapsto tu^2$ and $y\mapsto tu$.
	The relations reduce every power series to a $D$-linear combination
	of $1,x,y$, and their images are independent in $D[u]/(u^3)$.
	Thus $R$ is free of rank three over $D$, and $t$ is a
	nonzerodivisor.  Relations $x^2=0$ and $y^2=tx$ say that $y^4=0$, hence  $\p=(x,y)$ is nilpotent and $R/\p=D$;
	consequently $\dim R=1$ and $R$ is Cohen--Macaulay.
One has
$
	\m^2=(t^2,tx,ty)=t\m.
$
	The principal ideal $(t)$ is a reduction. Let us verify
	minimality directly. If a reduction $J$ of $\m$ is contained in
	$(t)$, write $J=tI$. For a sufficiently large $a\geq1$, the
	identities $\m^{a+1}=t\m^a=J\m^a=tI\m^a$ imply, by
	cancellation of $t$, that $\m^a=I\m^a$. If $I$ were proper,
	Nakayama's lemma would give $\m^a=0$, contrary to $t^a\ne0$.
	Hence $I=R$ and $J=(t)$. Also $\m\ne(t)$. A minimal reduction with reduction number zero would have
	to equal $\m$ itself, which is impossible since $\m$ contains the
	proper reduction $(t)$. Therefore $\rr(\m)=1$.
	Inverting $t$ gives 
$R[1/t]\cong K[y]/(y^3)$ with the convenience that $K:=k((t))$ and $x=t^{-1}y^2.$

	\begin{claim}\label{c}
This ring is  local, and its maximal ideal contracts to
$\p$. 	
\end{claim}

\begin{PROOF}{\ref{c}}	
Consider the localization map $R\to S:=R[1/t]$. One has
$(y)S\cap R=(x,y)=\mathfrak p.
$ To  this end, observe that
$R/\mathfrak p\cong k[[t]]$
and localizing at $t$ gives
$(R/\mathfrak p)[1/t]\cong k((t))=K.$
Thus the kernel of the composite map
$
R\to R[1/t]\cong K[y]/(y^3)\to/(y)\cong K
$
is exactly $\mathfrak p$. This kernel is also $(y)S\cap R$, since $(y)$ is the kernel of $S\to S/(y)$.
Therefore, the maximal ideal $(y)$ of $S$ contracts to $\mathfrak p$. Thus $R_{\mathfrak p}\cong R[1/t]$,  because both invert precisely the elements outside $\mathfrak p$, up to the resulting localization.	
	\end{PROOF}

Thanks to {Claim} \ref{c},  $R_{\p}=R[1/t]$. In an artinian local ring the zero
	ideal is a reduction of the nilpotent maximal ideal, so it is
	its unique minimal reduction. Here $(y)^2\ne0$ and $(y)^3=0$,
	which gives $\rr(\p R_{\p})=2$.
\end{proof}

\begin{proposition}\label{prop:largegap}
	For every integer $N\geq2$ there is a complete one-dimensional
	Cohen--Macaulay local ring $(R_N,\m_N)$ with a prime
	$\p_N\subsetneq\m_N$ such that
$\rr(\m_N)=1$ and $\rr(\p_N(R_N)_{\p_N})=N.
$ Moreover, the failure can be made \emph{arbitrarily large} by
	varying the ring, so no inequality of the form
	$\rr(\p R_{\p})\leq f(\rr(\q R_{\q}))$ with a fixed function $f$
	can hold in general.
\end{proposition}

\begin{proof}Let $D:=k[[t]]$.
	Inside $D[u]/(u^{N+1})$,  take the subring
$
R:=	R_N:=D\cdot1\;\oplus\;\bigoplus_{i=1}^N D\cdot(tu^i).
$
	It is finite free over $D$, hence Noetherian, and
	$\p:=\p_N:=\bigoplus_{i=1}^ND\cdot(tu^i)$ is nilpotent with quotient
	$D$. Its unique maximal ideal is $\m =tD+\p $. 	Thus $\mathfrak m $ is generated by
	$\{
	t,\ tu,\ tu^2,\ \dots,\ tu^N
	\}$.
	Squaring, the generators of $\mathfrak m ^2$ are all products of two of these:
	\[
	t^2,\qquad t\cdot tu^i=t^2u^i\ (1\le i\le N),\qquad
	(tu^i)(tu^j)=t^2u^{i+j},
	\]
	and $u^{i+j}=0$ whenever $i+j\ge N+1$. Hence
	$
	\mathfrak m ^2=(t^2,\ t^2u,\ t^2u^2,\ \dots,\ t^2u^N).
	$
	On the other hand
	\[
	t\mathfrak m =t\bigl(tD+\mathfrak p \bigr)
	=t^2D+\bigoplus_{i=1}^{N}D(t^2u^i)
	=(t^2,\ t^2u,\ t^2u^2,\ \dots,\ t^2u^N).
	\]
	Therefore
	$
{\mathfrak m ^2=t\mathfrak m .}
	$ The $t$-adic and $\m $-adic topologies agree,
	so $R $ is complete; freeness over $D$ makes it Cohen--Macaulay
	of dimension one. The same minimality argument as in
	Proposition~\ref{prop:small} gives $\rr(\m )=1$. Finally,
$
	(R )_{\p }=R [1/t]\cong k((t))[u]/(u^{N+1}),
$
	whose maximal ideal has reduction number $N$.
\end{proof}

	\begin{remark}
i) Proposition~\ref{prop:largegap} generalizes Proposition~\ref{prop:small}: the latter is the case $N=2$. Indeed, with $x\mapsto tu^2$, $y\mapsto tu$,
		\[
		R=k[[t,x,y]]/(x^2,xy,y^2-tx)\;\cong\;k[[t]]\cdot 1\oplus k[[t]](tu^2)\oplus k[[t]](tu)
		=D\cdot 1\oplus\bigoplus_{i=1}^{2}D(tu^i)\subseteq D[u]/(u^3),
		\]
		which is exactly $R_N$ for $N=2$.

		ii) In both cases $\mathfrak p$  is the nilpotent ideal with $R/\mathfrak p\cong D$, $\rr(\mathfrak m)=1$, and after inverting $t$,
		\[
		R[1/t]\cong K[u]/(u^3),\qquad (R_N)[1/t]\cong K[u]/(u^{N+1}),
		\]
		with reduction number $2$, resp.\ $N$, at the maximal ideal $(u)$. The pair shows that the gap $\rr(\mathfrak p R_{\mathfrak p})-\rr(\mathfrak m)$ is \emph{unbounded}.

iii) The rings we constructed were not  be reduced.
\end{remark}

\begin{observation}
If $(R,\fm)$ is 1-dimensional reduced ring, then for any	$\p\subset\m$,  
$
\rr(\p R_{\p})\leq \rr(\m).
$ 
\end{observation}
\begin{proof}We may assume $\p\subsetneq\m$.
	Since $R$ is reduced it satisfies Serre's condition $(R_0).$ But, $\dim(R)=1$, it follows that $R_{\p}$ is a field and so $\p R_{\p}=0$. Consequently, 
	$
	\rr(\p R_{\p})\leq \rr(\m).
	$
\end{proof}
\section{A finite bound for every fixed local ring}
\label{sec:uniform}
\subsection{A bound from parameter colength}
\label{sec:colength}

The next estimate works without a Cohen--Macaulay assumption (compare with \cite[Theorem 2.44]{Vasconcelos}).

\begin{lemma}\label{lem:generators}
	Let $(A,\n)$ have dimension $d\geq1$, and let
	$Q=(x_1,\dots,x_d)$ be a parameter ideal. Set
	$L=\ell_A(A/Q)$. For every integer $h\geq1$,
$
		\mu_A(\n^h)\leq h^{d-1}L+d-1.
$
\end{lemma}

\begin{proof}
	First suppose $d=1$. If $x=x_1$ and $I$ is any $\n$-primary
	ideal, and
	\[
	0\longrightarrow I\longrightarrow A\longrightarrow A/I\longrightarrow 0
	\]
	a short exact sequence of  $A$-modules. Apply multiplication by $x$ to each term:
	\[
	\begin{array}{ccccccccc}
		0 & \to & I & \to & A & \to & A/I & \to & 0\\[2pt]
		& & \downarrow{\scriptstyle x} & & \downarrow{\scriptstyle x} & & \downarrow{\scriptstyle x} & & \\[2pt]
		0 & \to & I & \to & A & \to & A/I & \to & 0
	\end{array}
	\]
	The snake lemma gives the exact \emph{kernel--cokernel sequence}
	\[
	0\longrightarrow (0:_I x)\longrightarrow (0:_A x)\longrightarrow (0:_{A/I}x)
	\longrightarrow I/xI\longrightarrow A/xA\longrightarrow (A/I)/x(A/I)\longrightarrow 0.
	\]
	Set $M=A/I$. For the endomorphism $x:M\to M$ of the finite-length module $M$,
\[
\len(M)=\len(\ker x)+\len(\operatorname{im}x)
=\len\bigl((0:_M x)\bigr)+\len(M)-\len\bigl(M/xM\bigr),
\]
hence
$
\len\bigl((0:_M x)\bigr)=\len\bigl(M/xM\bigr).
$
Thus the kernel term $\len\bigl((0:_{A/I}x)\bigr)$ and the cokernel term
$\len\bigl((A/I)/x(A/I)\bigr)$ are equal and cancel in the length identity. In order to all terms be of finite length, we rewrite:
	
	\[
0\longrightarrow (0:_A x)/(0:_I x)\longrightarrow (0:_{A/I}x)
\longrightarrow I/xI\longrightarrow A/xA\longrightarrow (A/I)/x(A/I)\longrightarrow 0.
\]	
Taking lengths along this exact sequence,
	\[
	\len(I/xI)-\len(A/xA)+\len\!\bigl((A/I)/x(A/I)\bigr)
	-\len\!\bigl((0:_{A/I}x)\bigr)+\len\!\bigl((0:_A x)\bigr)-\len\!\bigl((0:_I x)\bigr)=0.
	\]
This gives
\[
\ell_A(I/xI)
=\ell_A(A/xA)
-\ell_A\bigl((0:_A x)/(0:_I x)\bigr)
\leq\ell_A(A/xA).
\]
	Since $xI\subseteq\n I$, it follows that
	$\mu_A(I)\leq\ell_A(I/xI)\leq L$.
	Taking $I=\n^h$ proves the assertion for $d=1$.
	For $d>1$, we look at
$B:=A/(x_1^h,\dots,x_{d-1}^h)$ and $ I=\n^h.$
	Then $\dim B=1$, $x_d$ is a parameter of $B$, and
	$(x_1^h,\dots,x_{d-1}^h)\subseteq I$. Lifting generators from
	the image $IB$ and adjoining these $d-1$ elements gives
$$\mu_A(I) \leq\mu_B(IB)+d-1 
		 \leq\ell_A\bigl(A/(x_1^h,\dots,x_{d-1}^h,x_d)\bigr)+d-1.
$$
	To bound the last length, we increase the exponents of $x_1,\ldots,x_{d-1}$ one at a time. The key observation is that, for any finite-length $A$-module $C$ and any $x\in A$, the filtration
$C\supseteq xC\supseteq x^2C\supseteq\cdots\supseteq x^hC$
	has $h$ successive factors $x^jC/x^{j+1}C$, for $0\le j<h$. Each factor is a homomorphic image of $C/xC$, so
$\ell_A(C/x^hC)\le h\,\ell_A(C/xC).$
	Applying this inequality successively to $x_1,\ldots,x_{d-1}$, we obtain
	\[
	\ell_A\bigl(A/(x_1^h,\ldots,x_{d-1}^h,x_d)\bigr)
	\le h^{d-1}\ell_A(A/Q)=h^{d-1}L.
	\]
	Consequently, $\mu_A(\mathfrak n^h)\le h^{d-1}L+d-1$, as desired.
\end{proof}

We use the following form of the Eakin--Sathaye theorem
\cite[Theorem 8.6.8]{HS}:
\begin{fact}\label{eq:ES}
If $(A,\n)$ has infinite residue field and
$\mu_A(I^h)<\binom{h+d}{d},$ then there exist $a_1,\dots,a_d\in I$ such that
 	$I^h=(a_1,\dots,a_d)I^{h-1}$. For an $\n$-primary ideal in
 	dimension $d$, this is a minimal reduction. Here and below
 	$d\geq1$ in applications of this statement.
\end{fact}

\begin{corollary}\label{cor:localbound}
	If $(A,\n)$ has infinite residue field and dimension $d\geq1$,
	then for every parameter ideal $Q$,
	$
	\rr(\n)\leq d!\,\ell_A(A/Q)-1.
	$
	More generally, any integer $h\geq d!\,\ell_A(A/Q)$ satisfies
	$\rr(\n)\leq h-1$.
\end{corollary}

\begin{proof}
	Set $L=\ell_A(A/Q)$ and take $h\geq d!L$. If $d=1$,
	Lemma~\ref{lem:generators} gives
	$\mu_A(\n^h)\leq L\leq h<h+1$. For $d\geq2$, it gives
$$
	\mu_A(\n^h)\leq Lh^{d-1}+d-1
	\leq\frac{h^d}{d!}+d-1
	<\binom{h+d}{d}.
$$
	For the strict inequality, expand $(h+1)\cdots(h+d)/d!$.
	In addition to $h^d/d!$, its $h^{d-1}$ term alone is
	$\frac{d(d+1)}{2d!}h^{d-1}>d-1$, since $h\geq d!$.
	Now apply {Fact} \ref{eq:ES} with $I=\n$.
\end{proof}

\subsection{A finite bound for sup $\{\rr(\p R_{\p})\}$}\label{sec:uniform2}

Here is the precise standard uniformity input used below.
Huneke's uniform Brian\c con--Skoda theorem implies that (see \cite{Huneke}). The bar denotes integral closure of an ideal.
\begin{fact}\label{eq:BS}
If $U$ is a complete reduced Noetherian local ring, there is an
	integer $b\geq0$ such that
$
		\overline{I^{b+1}}\subseteq I$
 {for every ideal } $I\subseteq U$.
\end{fact}

\begin{lemma}\label{lem:uniformC}
	For a Noetherian local ring $(R,\m)$ there is an integer $C\geq1$
	such that, for every $\p\in\Spec R$ and every reduction $J$ of
	$\n=\p R_{\p}$,
$ \n^C\subseteq J.
$
\end{lemma}

\begin{proof}
	Let $T=\widehat R$, let $N=\sqrt{(0)_T}$, and choose $s\geq1$
	with $N^s=0$. Apply {Fact} \ref{eq:BS} to the complete reduced local
	ring $E=T/N$, obtaining $b$; set
	$
	C=s(b+1).
	$
	Fix $\p$ and $J$ as in the statement, put $S=R\setminus\p$,
	and set $U=S^{-1}T$. The map $R_{\p}\to U$ is
	faithfully flat, being a base change of $R\to T$.
	The reduction relation says that $\n$ is integral over $J$;
	this property persists in $U$ (see \cite[Proposition 8.1.1]{HS}) and in $U/NU=S^{-1}E$ (see \cite[Lemma 8.1.3~(3)]{HS}).
	Consequently, in $U/NU$ the $(b+1)$st power of the image of
	$\n U$ is contained in the integral closure of the $(b+1)$st
	power of the image of $JU$, hence in the image of $JU$ by
{Fact} \ref{eq:BS}. Thus
$	(\n U)^{b+1}\subseteq JU+NU.$
	Raising to the $s$th power yields
	\[
 \n ^CU\subseteq	(\n U)^C\subseteq(JU+NU)^s\subseteq JU+(NU)^s=JU.
	\]
	Contract along the faithfully flat map to obtain
the desired claim (see \cite[Theorem 7.5(ii)]{Mat}). 
\end{proof}

\begin{remark}\label{eq:edim}
We need a uniform embedding-dimension bound. By the Cohen
	structure theorem choose a surjection $V\twoheadrightarrow T=\widehat R$
	from a complete regular local ring $V$, and write $v=\dim V$.
	Then
$
		\edim(R_{\p})\leq v$ for all $\p\in\Spec R.$
	Indeed, choose a prime $P$ of $T$ over $\p$. The map
	$R_{\p}\to T_P$ is flat and local. Lech's
	embedding-dimension inequality for flat local maps gives
	$\edim(R_{\p})\leq\edim(T_P)$; see
	\cite{Lech}. If $\widetilde P$ is
	the preimage of $P$ in $V$, then $T_P$ is a quotient of the
	regular local ring $V_{\widetilde P}$, whose dimension is at
	most $v$. This proves the desired claim.
\end{remark}
  
  Our next main result is as follows.

\begin{theorem}\label{thm:finite}
	Let $(R,\m)$ be a Noetherian local ring of dimension $D\geq1$.
	Choose $C$ as in Lemma~\ref{lem:uniformC} and $v$ as in
Remark	\ref{eq:edim}. Finally, let
$	M:=\binom{v+C-1}{v} $ and $H:=D!M-1.$
Then, for every prime $\p$ for which $\kappa(\p)$ is infinite, we have
$ 
		\rr(\p R_{\p})\leq H.
$ 
	In particular, the supremum in Theorem~\ref{thm:answers}(ii)
	is finite, for arbitrary residue field of $R$.
\end{theorem}

\begin{proof}
	Fix such a $\p$, and set $(A,\n)=(R_{\p},\p R_{\p})$
	and $d=\dim A$. If $d=0$, $A$ is Artinian and the unique
	minimal reduction of $\n$ is zero. Lemma~\ref{lem:uniformC}
	therefore gives $\n^C=0$ and $\rr(\n)\leq C-1$.
	
	Suppose $d\geq1$. Since the residue field of $A$ is infinite,
	$\n$ has a minimal reduction $Q$ generated by $d$ elements;
	it is a parameter ideal. By {Lemma} \ref{lem:uniformC},
	$\n^C\subseteq Q$. Since $\mu_A(\n)\leq v$ (see Remark	\ref{eq:edim}), the graded
	pieces of the $\n$-adic filtration satisfy
$\ell_A(\n^i/\n^{i+1})\leq\binom{v+i-1}{i}.$
	Summing these inequalities, with the degree-zero term equal
	to one, gives
$$
		L:=\ell_A(A/Q)
		\leq\ell_A(A/\n^C)
		\leq\sum_{i=0}^{C-1}\binom{v+i-1}{i}
		=\binom{v+C-1}{v}=M.
$$
	Corollary~\ref{cor:localbound} now implies
	$\rr(\n)\leq d!L-1\leq D!M-1=H$.
	The dimension-zero case also satisfies this bound, because
	$v\geq1$ and hence $M\geq C$.
	
	Finally, every nonmaximal prime of a local ring has infinite
	residue field. Otherwise $R/\p$, which embeds into its finite
	fraction field $\kappa(\p)$, would be a finite domain and thus
	a field; this would make $\p$ maximal. If $R/\m$ is infinite,
$
\rr(\p R_{\p})\leq H.
$ If it is finite,
	only $\m$ is omitted, and Remark~\ref{rem:existence} gives
	$\rr(\m)<\infty$. In that case the supremum is at most
	$\max\{H,\rr(\m)\}$.
\end{proof}

If $\dim R=0$, then $\Spec R=\{\m\}$ and the assertion is
immediate from nilpotence of $\m$. This completes the proof
of Theorem~\ref{thm:answers}(ii).

\subsection{An infinite supremum without noetherianity}
\label{sec:nonnoetherian}

 \begin{definition}
 	Let $k$ be any field, $D = k[[t]]$, $K = k((t))$, and for $n \ge 1$ set
$A_n = D[u]/(u^{n+1}).$
 	Inside $\prod_{n\ge1} A_n$, define
 	$
 	B = D\cdot\mathbf 1 + \bigoplus_{n\ge1} A_n. 
 	$\end{definition}
 	Thus an element of $B$ is a sequence eventually equal to a single element of $D$. Write $e_n$ for the idempotent supported at coordinate $n$.
 	\begin{notation}
 	Set	$R = k\cdot\mathbf 1 + tB.$
 		\end{notation}
 	 The diagonal element $t$ is a nonzerodivisor in $B$ and in $R$. \begin{remark}For this subsection only, allow non-Noetherian rings, and use the same definitions of reduction and minimal reduction whenever they exist.\end{remark}

 \begin{observation}
 	The ring $R$ is quasi-local and non-Noetherian, with maximal ideal $\m := tB$.
 \end{observation}
 
 \begin{proof}
 	Every element of $R$ has the form $c + tb$ with $c \in k$ and $b \in B$. If $c \ne 0$, each coordinate is a unit in $A_n$, and its inverse sequence is eventually the inverse of a single unit of $D$. Write $c+tb = c(1 + c^{-1}tb)$. For any element $x$ since $(1+x)\bigl(1 - x(1+x)^{-1}\bigr) = (1+x) - x = 1$, we have
$	(1+x)^{-1} = 1 - x(1+x)^{-1}.$ Applying this with $x = c^{-1}tb$ and multiplying by $c^{-1}$  gives
 	\[
 	(c+tb)^{-1} = c^{-1}\bigl(1 - c^{-1}tb\,(1+c^{-1}tb)^{-1}\bigr)
 	= c^{-1} - c^{-2}tb\,(1+c^{-1}tb)^{-1}.
 	\]
 	Using $(1+c^{-1}tb)^{-1} = c\,(c+tb)^{-1}$, the second term becomes $c^{-1}tb\,(c+tb)^{-1}$, so
$
 	(c+tb)^{-1} = c^{-1} - c^{-1}tb\,(c+tb)^{-1},
$
 	i.e.\ $(c+tb)^{-1} - c^{-1} = -t\,c^{-1}\,b\,(c+tb)^{-1}$. Thus $(c+tb)^{-1} \in B$, and
 $ 
 		(c+tb)^{-1} - c^{-1} = -t\,c^{-1}\,b\,(c+tb)^{-1} \in tB.
$ 
 	So $c+tb$ is a unit of $R$. Since $R/tB \cong k$, this proves that $R$ is quasi-local with maximal ideal $\m = tB$.
For non-Noetherianity: multiplication by $t$ identifies
 $\m/\m^2 \cong B/tB$
 	as $k$-vector spaces. The classes of the $e_n$ are linearly independent in $B/tB$, so this space is infinite-dimensional. Hence $\m$ is not finitely generated.
 \end{proof}

 \begin{lemma}\label{512}
 	The primes of $R$ consist of $\m$, primes $\p_n$ for $n \ge 1$, and one additional prime $\p_\infty$. So,  $\dim R = 1$. Their localizations satisfy
$
 	R_{\p_n} \cong K[u]/(u^{n+1})$, and $R_{\p_\infty} \cong K.
$
 \end{lemma}
 
 \begin{proof}
 	Inverting $t$ yields
 	$
 	C := R[1/t] = B[1/t] = K\cdot\mathbf 1 + \bigoplus_{n\ge1} K[u]/(u^{n+1}).
 	$
 	Let $F$ be the direct-sum ideal in this expression. If a prime $P$ of $C$ omits some $e_n$, then localization forces $e_n = 1$ and $1 - e_n = 0$. It therefore selects the $n$th factor, whose unique prime is $(u)$. Let $P_n$ denote this prime of $C$. Then
 	$
 	C_{P_n} \cong K[u]/(u^{n+1}).
 	$
 	If $P$ contains every $e_n$, it contains $F$; since $C/F = K$, it must be $F$ itself. Every element of $F$ has finite support and is annihilated by $1 - \sum_{n\in E} e_n$ for a suitable finite set $E$. This annihilator lies outside $F$, so $C_F = K$.
 	Contracting $P_n$ and $F$ to $R$ defines $\p_n$ and $\p_\infty$, respectively, and proves the asserted localization formulas. All these primes omit $t$ and lie strictly below $\m$. Conversely, any prime of $R$ containing $t$ equals $\m$, because $\m^2 \subseteq tR$. We have therefore listed every prime. The primes of $C$ are all maximal, so $\dim R = 1$.
 \end{proof}

 \begin{corollary}
 	The reduction numbers are
$
 	\rr(\m) = 1$, $\rr(\p_n R_{\p_n}) = n$ and $\rr(\p_\infty R_{\p_\infty}) = 0.
 $
 	In particular, every displayed reduction number is finite and their supremum is infinite.
 \end{corollary}
 
 \begin{proof}
 	We have $\m^2 = t^2B = t\m$, and $\m \ne tR$ because $te_n \notin tR$. We verify that $tR$ is a minimal reduction without using a finite-generation assumption. If $J \subseteq tR$ is a reduction of $\m$, write $J = tI$ with $I$ an ideal of $R$. For some $a \ge 1$,
$
 	t^{a+1}B = \m^{a+1} = J\m^a = t^{a+1}IB.
 $
 	Cancellation inside $B$ gives $IB = B$. A proper ideal $I$ of the local ring $R$ lies in $tB$, which would force $IB \subseteq tB \ne B$. Hence $I = R$ and $J = tR$. Thus $\rr(\m) = 1$, since reduction number zero would force $\m$ itself to be minimal.
 	According to {Lemma} \ref{512}, localizing at $\p_n$, the maximal ideal is $(u)$, with nonzero $n$th power and zero $(n+1)$st power. Its unique minimal reduction is zero and its reduction number is $n$. At $\p_\infty$ the maximal ideal is zero and the reduction number is zero. Hence all displayed reduction numbers are finite, while their supremum over $n$ is infinite.
 \end{proof}
 
 \section{Loewy Length Conjecture}

We start with a modification of Conjecture \ref{14}.

\begin{question}
Let $(R,\m)$ be a noetherian local ring and $M$ a
finitely generated $R$-module of finite projective dimension with
$\m^nM=0$. Does $\m^n\subseteq(\underline{x})$ for some system of parameters $\underline{x}$?
\end{question}

 For every finite field $k:=\mathbb F_q$ there is a one-dimensional
hypersurface $R$ with residue field $k$ and a nonzero module of projective
dimension one for which the containment fails at $n=q+1$:

\begin{notation}
Let $A:=k[[x,y]]$, $f:=y\prod_{\lambda\in k}(x-\lambda y)$, and
$R:=R_q:=A/(f)$, with $\m=(x,y)R$. 
\end{notation}

Also, we need an auxiliary ring:

\begin{notation} 
Let	$K/k$ be a quadratic extension, $\alpha\in K\setminus k$ and set $S:=K[[x,y]]/(f)=R\otimes_k K$.\end{notation}

\begin{lemma}
	Any $S$-module of finite projective dimension,
	has  finite projective dimension as an $R$-module.
\end{lemma}

	\begin{proof}
	The ring $S $ is free of rank two over $R$. So, the claim is clear.\end{proof}
\begin{example}
There is  an $R$-module $M$ of finite projective dimension and Loewy length $q+1$.
\end{example}

\begin{proof}
	Let $u:=x-\alpha y$ and define $M:=S/(u)$. Substitution gives
$M\cong K[[y]]/(y^{q+1})$
and  $f(\alpha y,y)=y^{q+1}\prod_{\lambda\in k}(\alpha-\lambda).$
In $K[[x,y]]$, $u$ generates a prime ideal and $f\notin(u)$ as  $\alpha\in K\setminus k$.
If $ub=fc$, then $u\mid c$, and cancellation gives $b\in(f)$.
Thus multiplication by $u$ is injective on $S$, i.e., there is a short exact sequence
$0\to S\xrightarrow{\cdot u} S\to M\to 0$ of  $S$-modules.
In $K[[x,y]]/(u)$ we have $x\equiv \alpha y$, so substitution gives
$
f(\alpha y,y)=y^{q+1}\prod_{\lambda\in k}(\alpha-\lambda).
$
Because $\alpha\in K\setminus k$, none of the factors $\alpha-\lambda$ vanishes, hence $\prod_{\lambda\in k}(\alpha-\lambda)\neq 0$ is a unit in $K[[y]]$. Therefore
$M\cong K[[y]]/(y^{q+1}),$ as claimed.
Under this isomorphism, the class of $y$ maps to the generator $y$ of $K[[y]]/(y^{q+1})$, while the class of $x$ maps to $\alpha y$.
Since $M$ is an $R$-module via $R\to S$, the maximal ideal $\mathfrak m=(x,y)R$ acts through the images of $x$ and $y$ in $M$. But in $M$ the class of $x$ equals $\alpha$ times the class of $y$. Hence, as operators on $M$,
$
x=\alpha y, y=y.
$
Thus every element of $\mathfrak m$ acts as a multiple of $y$, and
$
\mathfrak m^2M=\mathfrak m(\mathfrak m M)=\mathfrak m(yM)=y(\mathfrak m M)=y^2M,
$
and inductively
$
\mathfrak m^nM=y^nM
$
for all $n\ge 1$. Since $M=K[[y]]/(y^{q+1})$, we have $y^{q+1}M=0$, and therefore
$
\mathfrak m^{q+1}M=y^{q+1}M=0.
$
Similarly $\mathfrak m^{q}M=y^{q}M\neq 0$, because $y^{q}\notin(y^{q+1})$ in $K[[y]]$. Thus,   Loewy length of $M$ is $q+1$.\end{proof}

For $q=2$,  $R$ is the ring of
Hashimoto and Shida \cite[Example 3.2]{HSh}. Bartels proved the generalized
Loewy length of $R_q$   is $q+2$, see \cite[Proposition 3.5]{Bartels}. This gives:

\begin{corollary}
Adopt the above notation. Then $\llength_R(M)=q+1\ngeq q+2 =\gll(R)$.
\end{corollary}
KC and Pollitz proved their inequality for infinite residue
fields when $\Gr_{\m}(R)$ is Cohen--Macaulay, see
\cite[Conjecture A and Theorem B]{NP}.

\section*{Acknowledgments}

The author is  grateful to  Josh Pollitz   and   Nawaj KC  for useful comments.

\end{document}